\documentclass[12pt]{article}
\pdfoutput=1
\usepackage{amsmath,amsfonts,amsthm,amssymb}
\usepackage{graphics,color,url}
\usepackage{graphicx}
\usepackage{epsfig}
\usepackage{fullpage}
\usepackage{euscript}

\graphicspath{{Figures/}}

\newtheorem{theorem}{Theorem}[section]

\newtheorem{lemma}[theorem]{Lemma}

\theoremstyle{remark}
\newtheorem*{remark}{Remark}

\newcommand{\R}{{\mathbb R}}
\newcommand{\Z}{{\mathbb Z}}
\newcommand{\GP}{{\mathfrak P}}

\newcommand{\Gn}{{\mathfrak n}}

\newcommand{\expo}{{\textup{exp}}}

\title{\textbf{Translational tilings by a polytope, with multiplicity}}

\author
{\\Nick Gravin\footnote{Division of Mathematical Sciences, School of
Physical and Mathematical Sciences, Nanyang Technological
University, Singapore. Email: {\tt ngravin@pmail.ntu.edu.sg,
rsinai@ntu.edu.sg, shir0010@ntu.edu.sg}}\footnotemark[1]
\footnotemark[2] \and \\Sinai Robins\footnotemark[1] \and \\Dmitry
Shiryaev\footnotemark[1] \footnote{St.Petersburg Department of
Steklov Mathematical Institute RAS, Russia.} }

\date{}

\begin{document}

\maketitle \thispagestyle{empty}
\begin{abstract}
We study the problem of covering $\R^d$ by overlapping translates of
a convex body $P$, such that almost every point of $\R^d$ is covered
exactly $k$ times. Such a covering of Euclidean space by translations is called a $k$-tiling.   The
investigation of tilings (i.e. $1$-tilings in this context) by translations began with  the work of Fedorov \cite{Fed}  and Minkowski~\cite{Min}.
Here we extend the investigations of Minkowski to $k$-tilings by proving that if a convex body
$k$-tiles $\R^d$ by translations, then it is centrally symmetric, and
its facets are also centrally symmetric.  These are the analogues of Minkowski's conditions for $1$-tiling polytopes. 
Conversely, in the case
 that $P$ is a rational polytope, we also prove that if $P$ is
centrally symmetric and has centrally symmetric facets, then $P$
must $k$-tile $\R^d$ for some positive integer $k$. 
\end{abstract}

\section{Introduction}

Suppose we are given a convex object $P$, and a multiset of discrete
translation vectors $\Lambda$.  We wish to cover all of $\R^d$ by
translating $P$ using the translation vectors in $\Lambda$, such
that each point $x \in \R^d$ is covered exactly $k$ times. Along the
boundary points of $P$ there may be some technical lower-dimensional
problems, but if we require that each point which does not lie on the boundary of any translate of $P$ 
to be covered exactly $k$ times, then we call such a
covering of $\R^d$ a $k$-tiling. The traditional field of tilings of
Euclidean space by translates of a single convex object $P$ has a
long and rich history. The usual notion of a tiling by translations
is thus equivalent  to the notion of a $1$-tiling. The reader is
invited to consult the books by Alexandrov \cite{Alex} and
Gruber~\cite{Gru} for a nice overview of the problem of tiling space
with translates of one convex body.     

Tilings of $\R^d$ by translations of a single object have been
extensively studied from as early as 1881~\cite{Fed}, by the
mathematical crystallographer Fedorov, and are an active research
area today.   For example, translational tilings of sets on the real line have
been studied in the $90$'s by Lagarias and Wang \cite{Jeff}.
There is also a beautiful  recent survey article on tilings in various different mathematical contexts, 
by Kolountzakis and Matolcsi~\cite{Kol3}.

We first note that if we have any $k$-tiling by a
convex object $P$, then it is an elementary fact that the convex
body $P$ must be a polytope, and we may therefore assume henceforth
that any convex object $P$ that $k$-tiles is a polytope.


Minkowski \cite{Min} has shown that if a convex body $P$ tiles
$\R^d$ by a lattice, then it follows that $P$ is a centrally
symmetric polytope, with centrally symmetric facets.  Venkov \cite{Ven} and
McMullen  \cite{McM} proved that if a convex body $P$ tiles $\R^d$ by
translation, then for each of its codimension two faces $F$ there
are either four or six faces which are translates of $F$.

Here we find analogues of the necessary Minkowski conditions in the
case of general $k$-tilings, for any integer $k$ (see the main
Theorem \ref{thm_main} below).

Despite the beautiful characterization of $1$-tilers, given
collectively by Minkowski, Venkov, and McMullen, there is
still no known complete classification of polytopes that admit a
$k$-tiling, even in two dimensions.  However, it is known that in
$\R^2$, every $k$-tiling convex body has to be a centrally symmetric
polygon.  Also, there exists a characterization by Bolle~\cite{Bol}
of all {\it lattice} $k$-tilings of convex bodies in $\R^2$.
Kolountzakis \cite{Kol2} proved that every $k$-tiling of $\R^2$ by a
convex polygon $P$, which is not a parallelogram, is  a $k$-tiling
with a finite union of two-dimensional lattices.

\begin{figure}
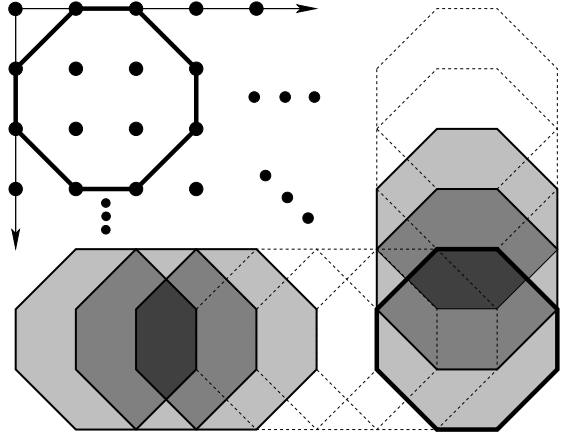

\begin{center}
\include{Figures/fig_8gon}
\end{center}
\caption{ An octagon that $7$-tiles, but does not $1$-tile. Here,
each point in the interior of the octagon is covered exactly $7$
times, once we translate the octagon by all of the integer
translation vectors.} \label{octagon}
\end{figure}

A {\bf parallelotope} is, by definition, a convex polytope that
tiles (i.e. $1$-tiles) $\R^d$ facet-to-facet, with a lattice.  That
is, its multiset of discrete translation vectors $\Lambda$ is in fact
given by a lattice in this case.   It was proved by McMullen that if
a polytope tiles $\R^d$  with a discrete multiset of translations
$\Lambda$, then it must also admit a facet-to-facet tiling with a
{\it lattice}.   In other words, McMullen showed that every
$1$-tiler must be a parallelotope.   A very active area of current research 
deals with the ``Voronoi conjecture'', which 

A {\bf zonotope} in $\R^d$ is a polytope which can be represented as
a Minkowski sum of finitely many line segments. Equivalently, a
zonotope is a polytope in $\R^d$ with the property that all of its
$k$-dimensional faces are centrally symmetric, for all $1 \leq k
\leq d$. For example, the zonotopes in $\R^2$ are the centrally
symmetric polygons.  A third equivalent definition for a zonotope is
that it is the projection of a $d$-dimensional cube, for some $d$.
For a good reference regarding these equivalences, and more about
polytopes, see the book by G. Ziegler \cite{Zie}.

It is clear that not all zonotopes are parallelotopes, an easy
example being furnished by the octagon (see fig.\,ref{octagon}) 
in two dimensions, which clearly does not tile by a lattice of translation vectors;
conversely, not all parallelotopes are zonotopes, as evidenced by
the example of the $24$-cell given below.  In fact, McMullen has
given a beatiful characterization of those parallelotopes which are
zonotopes, in terms of unimodular systems (see \cite{McM2} for more
details).

Very little is known about the precise classification of  polytopes
which $k$-tile $\R^d$ by translations.  We outline some specific
open questions in the last section that pertain to the current state
of affairs along these lines.  (see \cite{Gru}, pages 463-479 for
more details about $1$-tiling polytopes and some open problems).


\noindent We can now state the main result of this paper.
\begin{theorem}\label{thm_main}
If a convex polytope $k$-tiles $\R^d$ by translations, then it is
centrally symmetric and its facets are centrally symmetric.
\end{theorem}

We note that in $\R^3$ these two conditions are enough for a convex
body to necessarily be a zonotope. However, in dimension $4$ this is
no longer the case. A counterexample is furnished by the {\bf
$24$-cell}, which is a polytope in $\R^4$ which $1$-tiles $\R^4$, is
centrally symmetric, has centrally symmetric facets, but is not a
zonotope because it has $2$-dimensional faces that are triangles.
The $24$-cell is by definition the Voronoi region for the root
lattice $D_4$, and the reader may consult Coxeter \cite{Cox} for more
details.

Our proof of the main theorem above involves some new ideas that are quite different from Minkowski's proof for $1$-tilings. 
We also prove the following counter-part to the main Theorem above.

\begin{theorem}\label{thm_main2}
Every rational polytope $P$ that is centrally symmetric and has
centrally symmetric facets must necessarily $k$-tile $\R^d$ with a
lattice, for some positive integer $k$.

\medskip \noindent
Moreover, the polytope $P$ must $k$-tile $\R^d$ with the rational
lattice $\frac{1}{N}\Z^d$, where $N$ is the lcm of the denominators
of all the vertex coordinates of $P$.
\end{theorem}

The paper is organized as follows. Section \ref{MAIN} is devoted to
the proof of the main result, namely Theorem \ref{thm_main}, and
comprises the main body of the paper. Section \ref{second main theorem} is short, and is devoted to the proof of Theorem
\ref{thm_main2}. In section \ref{FourierSection} we provide a more analytic approach of the main result, using Fourier techniques. Although it is
not crucial to supply another proof of the main result, this
approach provides a Fourier  lens through which we can view our results.  

Kolountzakis has also studied this problem using the
Fourier approach, and indeed our Fourier approach borrows some
techniques from his work. In section \ref{solid-angles} we give
another necessary and sufficient condition for a polytope to
$k$-tile $\R^d$, this time in terms of the solid angles of the
vertices of $P$. Finally, in section \ref{open}  we mention some of
the important open problems concerning polytopes that $k$-tile
$\R^d$.

\section{Definitions and preliminaries}

We adopt the usual conventions and notation from combinatorial
geometry.  First, we recall that the {\bf Minkowski sum} of two
multisets $A \subset \R^d$ and $B \subset \R^d$ is the set
$A+B=\{a+b: a\in A, b\in B\}$, and that the Minkowski difference is
defined similarly by $A-B=\{a-b: a\in A, b\in B\}$.

For any set $A \subset \R^d$, its { \bf opposite set}  is defined as
$-1 \cdot A=\{-a: a\in A\}$.  We are particularly interested in the
case that both $A$ and $B$ are polytopes.  We are also keenly
interested in the case that $A$ is a polytope and $B$ is a discrete
set of vectors, so that here $A+B$ is a set of translated copies of
the polytope $A$.

Given a convex body $P\subseteq\R^d$, $ \partial P$ denotes the
boundary of $P$. The standard convention for $\partial P$ includes
the fact that it has ($d$-dimensional) Lebesgue measure $0$, with
respect to the Lebesgue measure of $\R^d$. We let the  interior of a
body $P$ be denoted by $\textup{Int}(P)$.  Throughout the paper,
$\Lambda$ denotes an infinite discrete multiset of vectors in $\R^d$,
which is not necessarily a lattice.

We say that body $P$ {\bf $k$-tiles} $\R^d$ with the discrete
multiset $\Lambda$, if after translating $P$ by each vector
$\lambda\in\Lambda$, almost every point of $\R^d$ (except for the
boundary points of translated copies of $P$) is covered by exactly
$k$ of these translated copies of $P$. This condition can be written
more concisely as follows:
$$
\sum\limits_{\lambda\in\Lambda}1_{P+\lambda}(v)=k,
$$
for all ${v}\notin \partial P+\Lambda$.

We also recall that a  { \bf facet} of a $d$-dimensional polytope is
any one of its $(d-1)$-dimensional faces. We let $V^{k}(F)$ denote
the $k$-dimensional volume of a $k$-dimensional object $F$, even if
$F$ resides in a higher dimensional ambient space, and sometimes we
simply write $V(F)$ for the $d$-dimensional volume of a
$d$-dimensional object $F \subset \R^d$. Finally,  $\#(A)$ denotes
the cardinality of any finite multiset $A$.


\section{Proof of Theorem \ref{thm_main}} \label{MAIN}

To simplify the ensuing notation, we will assume that $-1 \cdot P$
$k$-tiles $\R^d$.   We do not lose any generality, because $-1 \cdot
P$ $k$-tiles $\R^d$ if and only if $P$ also $k$-tiles $\R^d$.

We say that $v\in \R^d$ is in {\it general position} if there are no
points of $\Lambda$ on the boundary of $P+v$.  In other words, $v
\notin \Lambda-\partial P$. We first prove the following elementary
but useful lemma, giving an equivalent condition for $k$-tiling in
terms of the number of $\Lambda$-points that lie in a 'typical'
translate of $P$.

\begin{lemma}\label{lm_restate}
A convex polytope $-1 \cdot P$ $k$-tiles $\R^d$ by translations with
a multiset $\Lambda$
if and only if\\
 $\#(\Lambda\cap\{P+v\})=k$ for every $v$ in general position.
\end{lemma}

\begin{proof} Suppose that $-1 \cdot P$ $k$-tiles $\R^d$. Then for every
$v\notin \partial(-1 \cdot P)+\Lambda$ we can write
$$
k=\sum\limits_{\lambda\in\Lambda}1_{ \{-1 \cdot P+\lambda \} }(v)
=\sum\limits_{\lambda\in\Lambda}1_{P+v}(\lambda)=\#(\Lambda\cap\{P+v\}).
$$
It remains to mention that $\partial(-1 \cdot
P)+\Lambda=\Lambda-\partial P$. The proof in the other direction is
identical.
\end{proof}

\bigskip
We need to introduce some useful and natural notation for the
theorems that follow. Let $\GP$ be the vector space of the real
linear combinations of indicator functions of all convex polytopes
in $\R^d$. Thus, for example, if $P$ is any convex $k$-dimensional
polytope and $Q$ is any convex $m$-dimensional polytope, then
$\frac{1}{3} \cdot 1_P - 2 \cdot 1_Q \in \GP$.

One of the most important operators for us is the following boundary
operator, with respect to a vector $n\in\R^d$.   It is the function
$\partial_n:\GP\rightarrow\GP$, defined as follows:
\[
\partial_{n}1_P=1_{F^{+}}-1_{F^{-}},
\]
where $F^{+}$ and $F^{-}$ are the (possibly degenerate) facets of
$P$ with outward pointing normals $n$ and $-n$, respectively.   It
is a standard vector space verification that this operation is also
well-defined on $\GP$.

We also define this boundary operator on all of $\GP$, by letting it
act as a linear operator on the linear combinations of indicator
functions of polytopes.   For example, another iteration of this
operator on $P \subset \R^3$  yields $\partial_{n_2} (\partial_{n_1}
P) =  \partial_{n_2}( 1_{F^+}-1_{F^-}   )  = (1_{E_1} - 1_{E_2})
-(1_{E_3} - 1_{E_4})$, where $E_1, E_2$ are the edges (which are by
definition the $1$-dim'l faces) of $F^+$, and $E_3, E_4$ are the
edges of $F^-$.  In this case, each of the four edges is orthogonal
to both of the vectors $n_1$ and $n_2$, as is seen in Figure
$\ref{boundary.operator}$ below.

\begin{figure}
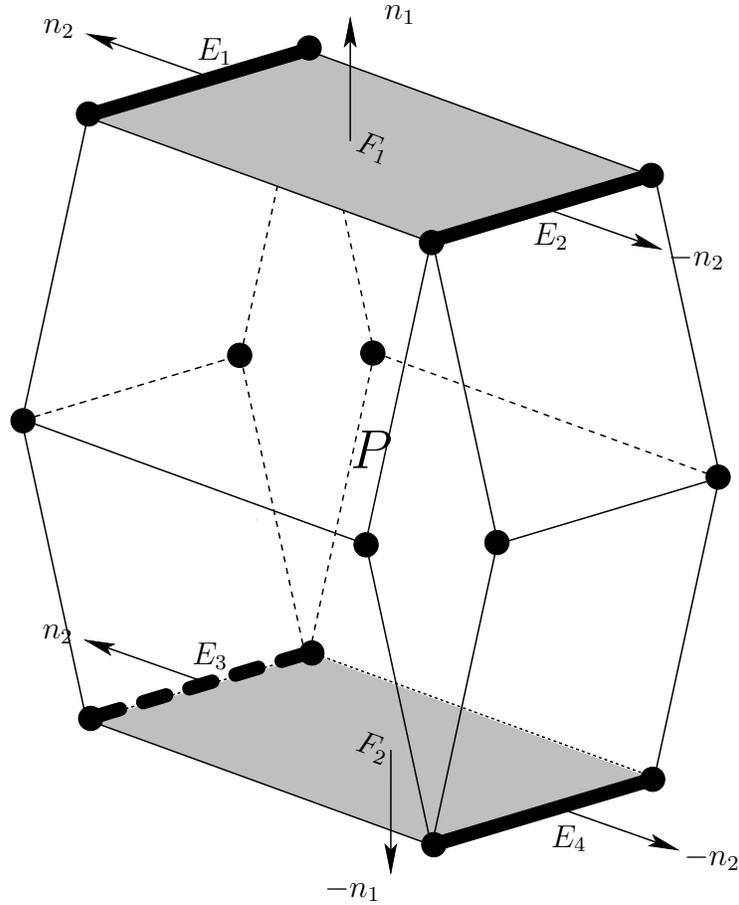

\begin{center}
\include{Figures/fig_2normals}
\caption{ The boundary operator with respect to $n_1$ picks out the
two facets $F^+$ and $F^-$, illustrating the definition of
$\partial_{n}1_P=1_{F^{+}}-1_{F^{-}}$.   A second iteration of the
boundary operator, this time with respect to $n_2$, picks out the
four edge vectors $E_1, E_2, E_3$, and $E_4$, thus visually
illustrating the identity $
\partial_{n_2} (\partial_{n_1} P) =  \partial_{n_2}( 1_{F^+}-1_{F^-}   )  = (1_{E_1} - 1_{E_2}) -(1_{E_3} - 1_{E_4}).
$ } \label{boundary.operator}
\end{center}
\end{figure}

For the sake of convenience, we also define the action of the
boundary operator $\partial_n$ on convex polytopes $P$ as follows:
\[
\partial_{n}P=supp(\partial_{n}1_P)=\{v\in\R^d|\partial_{n}1_P(v)\neq 0\},
\]
so that the same symbol now acts on the subset $P$.  However, we
note that the more salient operator for our discussions is still
$\partial_{n}1_P $.
 It is useful to
utilize both of these actions, the first being an action on
indicator functions, and the second being an action on subsets of
points $P \subset \R^d$.

We call a sequence $\Gn=(n_1,\ldots,n_m)$ of vectors in $\R^d$ an
orthogonal frame if they are pairwise orthogonal to each other.  We
denote it by $\Gn^{\bot}$ the subspace of $\R^d$ consisting of those
vectors which are orthogonal to every vector in the orthogonal frame
$\Gn$.

We define $\partial_{\Gn} := \partial_{n_m}\ldots\partial_{n_1}$, a
composition of boundary operators that is read from right to left.
In case $m=0$, when an orthogonal frame $\Gn$ is empty,
we define $\partial_{\Gn}$ to be an identity operator. Similarly to $\partial_{n}P$ we define a boundary operator
relative to  a whole frame $\Gn=(n_1,\ldots,n_m)$:

\[
\partial_{\Gn}P=supp(\partial_{\Gn}1_P)=\{v\in\R^d|\partial_{\Gn}1_P(v)\neq 0\}.
\]

Note that all the faces whose indicator functions appear in
$\partial_{\Gn}1_P$ must have codimension $m$, must be parallel to
each other, and must have outward pointing normals $n_m$ or $-n_m$
in $\partial_{n_{m-1}}\ldots\partial_{n_1}P$.

We can now separate $\partial_{\Gn}$ into two parts:
$\partial_{\Gn}^{+}$ and $\partial_{\Gn}^{-}$, corresponding to
faces with outward normals $n_m$ or $-n_m$, so that
$\partial_{\Gn}1_P = \partial_{\Gn}^{+}1_P - \partial_{\Gn}^{-}1_P$.
In other words, if $\partial_n 1_P=1_{F^{+}}-1_{F^{-}}$, then by
definition $\partial_n^{+}1_P=1_{F^{+}}$, and
$\partial_n^{-}1_P=1_{F^{-}}$.

We say that $v \in \R^d$ is in general position w.r.t. the
orthogonal frame $\Gn$, if there are no points of $\Lambda$ on any
boundary component of $\partial_{\Gn}(P+v)$. A more formal
description which we will have occasion to use below is that
$v\notin \Lambda-\partial\partial_{\Gn}P$.

Even though we only need to consider orthogonal frames of size
at most two in order to prove the theorem~\ref{thm_main}, we will prove
two following lemmas in general case, for an orthogonal frame of any size.

\begin{lemma}\label{lm_discrete}
Suppose $\#(\Lambda\cap\{P+v\})=k$ for every $v$ in general
position. Let $\Gn=(n_1,\ldots,n_m)$ be an orthogonal frame in
$\R^d$. Then for any $v$ in general position w.r.t $\Gn$ the
following formula holds:
\begin{equation}
\sum\limits_{\lambda\in\Lambda}\partial_{\Gn}1_{P+v}(\lambda)=0.
\label{fml_discrete}
\end{equation}
\end{lemma}
\begin{proof}
We proceed by induction on $m$. We remark that for $m=0$ the hypothesis tells us that
$\sum\limits_{\lambda\in\Lambda}\partial_{\Gn}1_{P+v}(\lambda)=k$, and for $m=0$ this operator is by definition the identity operator.
However, for each
 $m\geq 1$, we will show that $\sum\limits_{\lambda\in\Lambda}\partial_{\Gn}1_{P+v}(\lambda)=0$.



Suppose that for an $(m-1)$-dimensional orthogonal frame
$\Gn^{\prime}=(n_1,\ldots,n_{m-1})$ and for every $v$ in general
position w.r.t. $\Gn^{\prime}$ the formula holds:
$$\sum\limits_{\lambda\in\Lambda}\partial_{\Gn^{\prime}}1_{P+v}(\lambda)=const$$

Now consider any $m$-dimensional orthogonal frame
$\Gn=(n_1,\ldots,n_m)$, and $v$ in general position w.r.t $\Gn$. We
know that all $\Lambda$-points of $v+\partial_{\Gn}P$ lie in
$v+Int(\partial_{\Gn}P)$. Therefore, one can pick sufficiently small
$\epsilon'$, such that no $\epsilon'$-perturbation of $v$ by a
vector in $\Gn^{\bot}$ removes or adds any $\Lambda$-points to
$v+\partial_{\Gn}P$. Clearly, by doing so we do not change
$\sum\limits_{\lambda\in\Lambda}\partial_{\Gn}1_{P+v}(\lambda)$. On
the other hand, we may choose an $\epsilon'$-perturbation,
$v_{\epsilon'}$, such that all $\Lambda$-points in
$v_{\epsilon'}+\partial\partial_{\Gn^{\prime}}P$ get either inside
or outside of $v_{\epsilon'}+\partial_{\Gn^{\prime}}P$ (see
fig.\,\ref{fig:boundary1}).


\begin{figure}
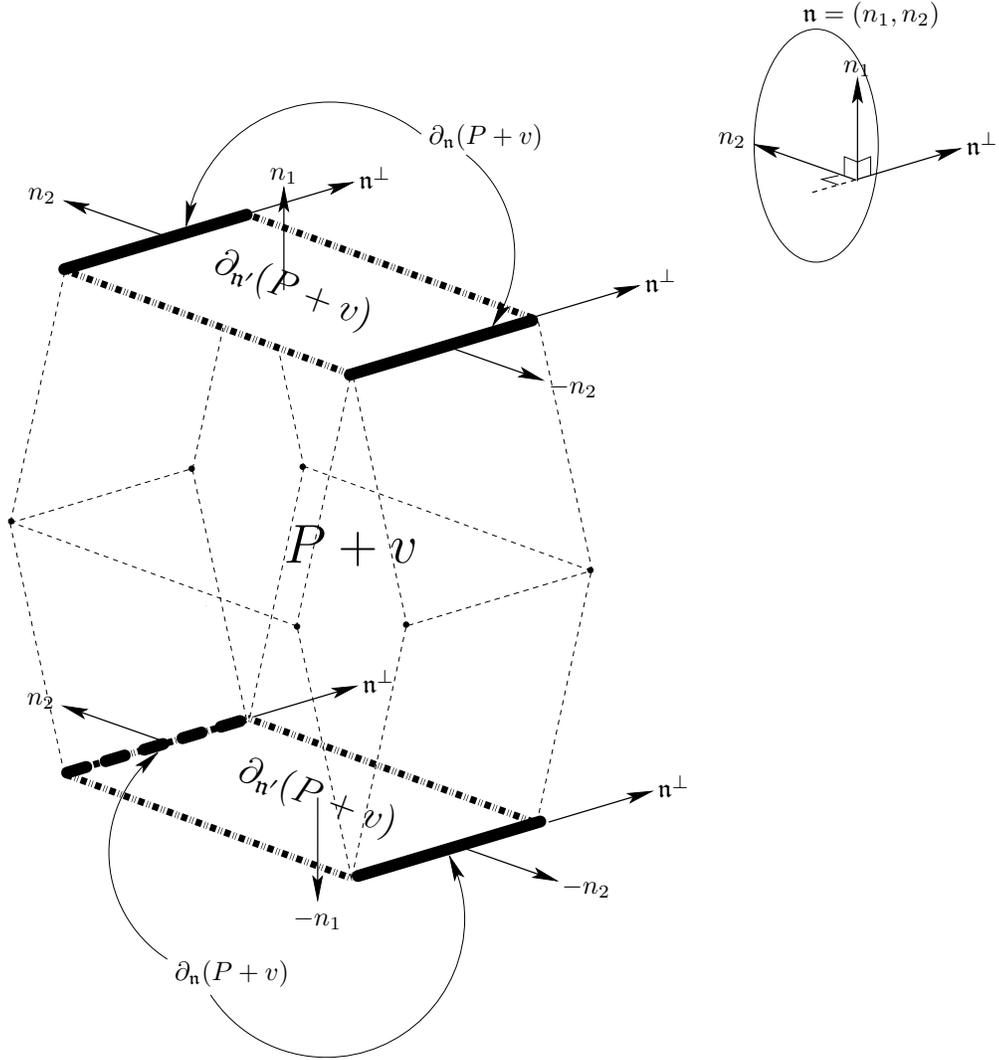

\begin{center}
\include{Figures/fig_boundary1}
\end{center}
\caption{ The $\epsilon'$ perturbation, along the $n^{\bot}$ direction, insures that all $\Lambda$ points have been removed from the four dotted edges on the upper facet and lower facet of $P + v$, giving us the set  $\partial_{\Gn^{\prime}}(P+v)$.    Also, the $\epsilon$ perturbation, along the 
$n_2$ direction, insures that all $\Lambda$ points on the right-hand bold edges, attached to the normal vector $-n_2$, will end up outside of the perturbed set, and that all $\Lambda$ points on the left-hand bold edges, attached to the normal vector $n_2$, will end up inside the perturbed set.   } 
\label{fig:boundary1}
\end{figure}

Then consider two small perturbations of $v_{\epsilon'}$ in the directions $n_m$ and
$-n_m$: $v_{\epsilon'}^{+}=v_{\epsilon'}+\epsilon n_m$ and $v_{\epsilon'}^{-}=v_{\epsilon'}-\epsilon n_m$, such that
$v_{\epsilon'}^{+}$ and $v_{\epsilon'}^{-}$ are in general position w.r.t $\Gn^{\prime}$,
and $\epsilon$ small enough so that there are no points of $\Lambda$
that lie in $P+v_{\epsilon'}^{\pm}$ and do not lie in $P+v_{\epsilon'}$ (such an $\epsilon$ can
be found, because $\Lambda$ is discrete).

By induction,
$\sum\limits_{\lambda\in\Lambda}\partial_{\Gn^{\prime}}1_{P+v_{\epsilon'}^{+}}(\lambda)=const=
\sum\limits_{\lambda\in\Lambda}\partial_{\Gn^{\prime}}1_{P+v_{\epsilon'}^{-}}(\lambda)$.

On the other hand, recalling that by definition $\partial_{\Gn}P=\partial_{n_m}\partial_{\Gn'}P$,
\begin{align}
\sum\limits_{\lambda\in\Lambda}\partial_{\Gn^{\prime}}1_{P+v_{\epsilon'}^{+}}(\lambda)-
\sum\limits_{\lambda\in\Lambda}\partial_{n_m}^{+}\partial_{\Gn'} 1_{P+v_{\epsilon'}}(\lambda)
&=
\sum\limits_{\lambda\in\Lambda}\partial_{\Gn^{\prime}}1_{P+v_{\epsilon'}}(\lambda)\cdot1_{Int(\partial_{\Gn^{\prime}}P)}(\lambda) \\
&=
\sum\limits_{\lambda\in\Lambda}\partial_{\Gn^{\prime}}1_{P+v_{\epsilon'}^{-}}(\lambda)-
\sum\limits_{\lambda\in\Lambda}\partial_{n_m}^{-}\partial_{\Gn'} 1_{P+v_{\epsilon'}}(\lambda).
\end{align}

It follows that $\sum\limits_{\lambda\in\Lambda}\partial_{n_m}^{+}\partial_{\Gn'}
1_{P+v}(\lambda)=\sum\limits_{\lambda\in\Lambda}\partial_{n_m}^{+}\partial_{\Gn'}
1_{P+v_{\epsilon'}}(\lambda)= \sum\limits_{\lambda\in\Lambda}\partial_{n_m}^{-}\partial_{\Gn'}
1_{P+v_{\epsilon'}}(\lambda)= \sum\limits_{\lambda\in\Lambda}\partial_{n_m}^{-}\partial_{\Gn'}
1_{P+v}(\lambda)$, which gives us (\ref{fml_discrete}), since
$\partial_{\Gn}=\partial_{n_m}^{+}\partial_{\Gn'}-\partial_{n_m}^{-}\partial_{\Gn'}$.
\end{proof}




For any polytope in $\R^d$ lying in an affine subspace parallel to
$\Gn^{\bot}$, we may consider a naturally defined lower dimensional
volume $\text{Vol}_{\Gn}$.  For example, if $d=3$ and $\Gn=(n_1, n_2)$, we
get $\text{Vol}_{\Gn}$ to be just a length of a line segment in $\R^3$. As
we know $\partial_{\Gn}1_P$ is a finite sum of indicator functions of
polytopes lying in affine subspaces parallel to $\Gn^{\bot}$ taken
with $+$ or $-$ signs. For each such indicator function $1_F$ let us denote
by $\text{Vol}_{\Gn}(1_F)$ the volume of polytope $F$. Note that we can
take any measurable object in the affine subspace parallel to
$\Gn^{\bot}$ instead of $F$. 

We now extend the notion of $\text{Vol}(S)$ to a more general notion of a signed linear combination of volumes.
We let $V_{\Gn}(\partial_{\Gn}1_P)$ denote the sum of the corresponding
volumes taken with different signs, and in a similar way we can write
$V_{\Gn}$ for any sum of positive and negative indicator
functions.  The next Lemma extends equality (\ref{fml_discrete}) in  Lemma \ref{lm_discrete} from a discrete measure of facets to a continuous measure of facets.

\begin{lemma}\label{lm_integral}
Under the same assumptions of lemma~\ref{lm_discrete}, the following
formula holds:
$$  V_\Gn( \partial_{\Gn}1_P) =0.$$
\end{lemma}
\begin{proof}

Let us recall what we have so far. Lemma~\ref{lm_discrete} tells us
$\sum\limits_{\lambda\in\Lambda}\partial_{\Gn}1_{P+v}(\lambda)=0$
for any $v$ with $v+\partial\partial_{\Gn}P$ containing no $\Lambda$
points.

For each $\lambda\in\Lambda$ let us consider a set $S$ of vectors $v$ enjoying the property that
 $\lambda\in v+\partial_{\Gn}P$ and $\Lambda\cap
\{v+\partial\partial_{\Gn}P\}=\emptyset$.   We call the set $S$
$\Gn$-interior w.r.t. $\lambda$.  We can also realize the set $S$ by excluding a finite number of  lower dimensional polytopes (polytopes $F$ with
$V_{\Gn}(F)=0$) from $\lambda-\partial_{\Gn}P$.  We call a vector $\Gn$-internal if it belongs to
$\Gn$-interior for some $\lambda\in\Lambda$.

Assume now that $V_{\Gn}(\partial_{\Gn}1_P)=A_1\neq 0$. Let us
also write $V_{\Gn}(|\partial_{\Gn}1_P|)=A_2\ge|A_1|>0,$ where by
$|\partial_{\Gn}1_P|$ we imply the sum of indicators of
$\partial_{\Gn}1_P$ with all negative coefficients of indicators
switched to their absolute value.

For any $R>0$ we may consider a ball $B_R$ in $\R^d$ with the center
at origin and given radius $R$. Clearly, there is a constant
$C=C(P)$, such that $B_R + (-1)\partial_{\Gn}P\subset B_{R+C}$ (see fig.\,\ref{fig:balls}). For
any positive real $R$ we define $N(R):=\#\{B_R\cap\Lambda\}$. For
each $\Gn$-internal $v\in B_R$ we may rewrite the formula from lemma
\ref{lm_discrete} and get
$$\sum_{\lambda\in
B_{R+C}\cap\Lambda}\partial_{\Gn}1_{\lambda-P}(v)=0.$$ This implies
$$V_{\Gn}\left(1_{B_R}\cdot\sum_{\lambda\in
B_{R+C}\cap\Lambda}\partial_{\Gn}1_{\lambda-P}\right)=0.$$

Also we know that $$\left|V_{\Gn}\sum_{\lambda\in
B_{R+C}\cap\Lambda}\partial_{\Gn}1_{\lambda-P}\right|=N(R+C)|A_1|.$$

\begin{eqnarray*}
\left|V_{\Gn}\left(\sum_{\lambda\in
B_{R+C}\cap\Lambda}\partial_{\Gn}1_{\lambda-P}\right)\right| &\leq&
\left|V_{\Gn}\left(1_{B_R}\cdot\sum_{\lambda\in
B_{R+C}\cap\Lambda}\partial_{\Gn}1_{\lambda-P}\right)\right|+
\\
&~&
\left|V_{\Gn}\left(\left(1_{B_{R+2C}}-1_{B_{R}}\right)\cdot\sum_{\lambda\in
B_{R+C}\cap\Lambda}
\partial_{\Gn}1_{\lambda-P}\right)\right|=
\\
&~&
\left|V_{\Gn}\left(\left(1_{B_{R+2C}}-1_{B_{R}}\right)\cdot\sum_{\lambda\in\left(B_{R+C}\setminus
B_{R-C}\right)\cap\Lambda}\partial_{\Gn}1_{\lambda-P}\right)\right|\leq
\\
&\leq& V_{\Gn}\sum_{\lambda\in\left(B_{R+C}\setminus
B_{R-C}\right)\cap\Lambda}|\partial_{\Gn}1_{\lambda-P}|=A_2\cdot\left(N(R+C)-N(R-C)\right).
\end{eqnarray*}

\begin{figure}
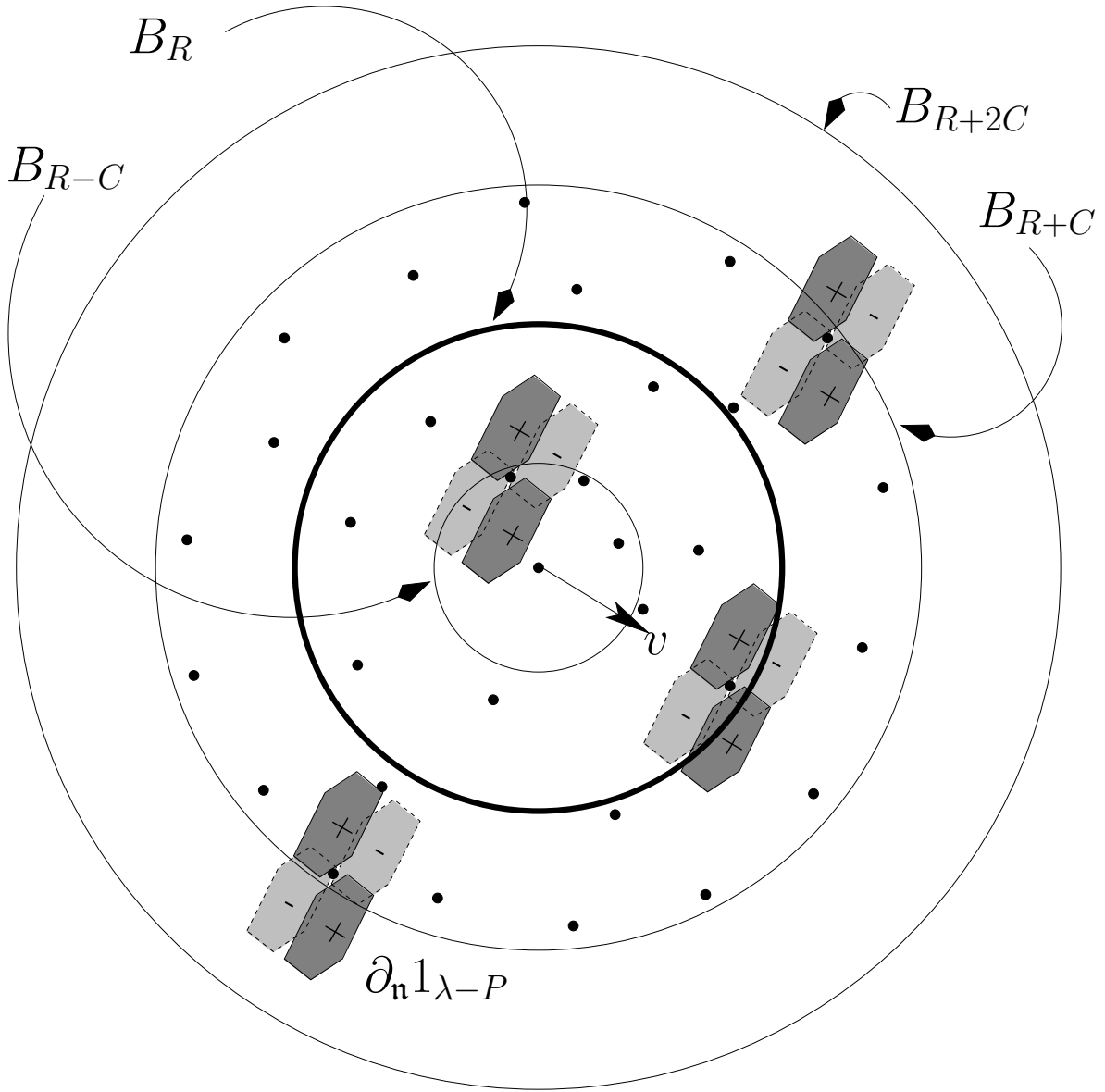

\begin{center}
\include{Figures/fig_balls_new}
\end{center}
\caption{None of the $\Lambda$-translates of $\partial_{\Gn}1_{\lambda-P}$ can overlap more than two adjacent shells between the concentric balls. }
\label{fig:balls}
\end{figure}

Thus we get $(1-\frac{|A_1|}{A_2})N(R+C)\geq N(R-C)$, which
establishes an exponential grows of $N(R)$ in $R$. We can cover
$B_R$ by a disjoint union of $O(R^{2d})$ cubes whose side-length is $\frac{1}{R}$. Thus taking sufficiently large $R$ we can find a
cube $K$ with side-length $\frac{1}{R}$, which contains more than $k$ $\Lambda$-points. We can now translate $P$ so that the cube $K$ is contained in $P$, and therefore this translate of $P$ now contains more than $k$ $\Lambda$-points, a contradiction. 

\end{proof}

In order to finish the proof of main theorem, we need the following
theorem by Minkowski~\cite{Min}.
\begin{theorem}[Minkowski]\label{thm_mink}
Convex polytope in $\R^d$ with given facet normals and facet
$(d-1)$-volumes is unique up to translation.
\end{theorem}

\bigskip
\begin{proof}[Proof of Theorem \ref{thm_main}]
We will first prove that $P$ is centrally symmetric. Take any pair of
facets of $P$, $F^{+}$ and $F^{-}$, with outward normals $n$ and
$-n$ respectively. Applying lemma~\ref{lm_integral} to $\Gn=(n)$ we
get $V_{\Gn}( \partial_{\Gn}1_P) =0$, which means that
$V(F^{+})=V(F^{-})$. Since $n$ can be chosen arbitrarily, polytopes
$P$ and $(-1)\cdot P$ have equal codimension $1$ volumes of facets
in every direction. By theorem~\ref{thm_mink} we get that
$P=(-1)\cdot P+v$ for some translation vector $v$, so $P$ is
centrally symmetric.

Similarly we prove that everly facet of $P$ is centrally symmetric.
Given a pair of opposite facets $F_1$ and $F_2$ of $P$ with outward
normals $n_1$ and $-n_1$ respectively, consider any direction
$n_2\in (n_1)^{\bot}$ and two pairs of corresponding faces of
codimension $2$: $F_1^{+}$ and $F_1^{-}$ are facets of $F_1$ with
outward normals $n_2$ and $-n_2$ respectively, $F_2^{+}$ and
$F_2^{-}$ are facets of $F_2$ with outward normals $n_2$ and $-n_2$
respectively. Applying lemma~\ref{lm_integral} to $\Gn=(n_1, n_2)$
we get $V_{\Gn}( \partial_{\Gn}1_P) =0$, which means that
$(V(F_1^{+})-V(F_1^{-}))-(V(F_2^{+})-V(F_2^{-}))=0$. But since $P$
is centrally symmetric, $F_1^{+}$ and $F_2^{-}$ are symmetric to
each other as well as $F_1^{-}$ and $F_2^{+}$, so
$V(F_1^{+})=V(F_2^{-})$ and $V(F_1^{-})=V(F_2^{+})$. 

Combining the last
three equations we get an equality for codimension $2$ faces of $P$:  $V(F_1^{+})=V(F_1^{-})$.  It follows that as $(d-1)$-dimensional objects, $F_1$ and $(-1)\cdot F_1$, themselves have equal facets in every direction (in their affine span), and again by
theorem~\ref{thm_mink} we get that $F_1$ is centrally symmetric. But
since $F_1$ could be chosen arbitrarily among the facets of $P$, every facet of $P$ is
centrally symmetric, which concludes the proof of
theorem~\ref{thm_main}.
\end{proof}

\begin{remark}
We note that  Lemma \ref{lm_discrete} gives us interesting information about the relationship between the $\Lambda$ points that lie in various faces, for any frame that has more than $2$ vectors.   In contrast, Lemma \ref{lm_integral} does not give us any additional information about the codimension $3$ volumes (or higher codimension volumes).   It is for this reason that we cannot conclude that codimension $3$ faces of a $k$-tiling polytope are centrally symmetric, and in fact they are not in general centrally symmetric, as the example of the $24$-cell shows.  
\end{remark}


\section{Proof of Theorem \ref{thm_main2} }\label{second main theorem}

\begin{proof}
We may assume, without loss of generality,  that our rational
polytope $P$ is an integer polytope, by dilating it by the lcm of
the denominators of all of the rational coordinates of its vertices. 
Now, given that $P$ has integer vertices, we will show that the polytope $P$ $k$-tiles $\R^d$ with $\Lambda=\Z^d$.

We claim  that in every general position $P$ has an equal number of
integer points on every pair of opposite facets. Indeed, since it is
centrally symmetric and has centrally symmetric facets (and integer
vertices), any two opposite facets are translations of one another
by some integer vector. It follows that for every integer point on a
facet there is a corresponding integer point on an
opposite facet, so their numbers are equal.

Now, consider any two general positions of $P$, say $P+u$ and $P+v$. 
There exists some path from $u$ to $v$ such that when we translate $P$ along this
path, no integer point of $\Z^d$ collides with any co-dimension $2$ face of the translates of $P$ along this path (see fig.\,\ref{moving.polygon}).
But
since in any general position the number of integer points on two
opposite facets of $P$ are equal, it follows that the number of points
inside $P$ along this path is constant. We conclude that any two
general positions of $P$ have the same number of interior integer points,
say $k$. Thus, $-P$ $k$-tiles $\R^d$ with the lattice $\Z^d$.
\end{proof}

\begin{figure}
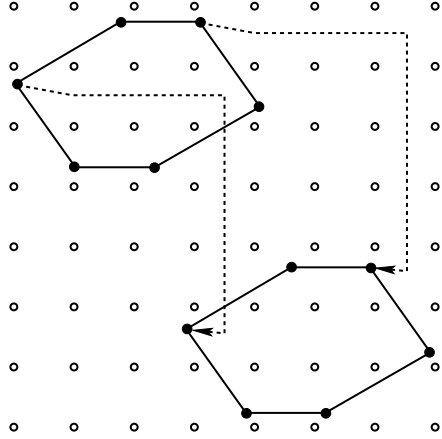

\begin{center}
\include{Figures/fig_path}
\end{center}
\caption{ This polygon illustrates that fact that there is always a
continuous path that a polygon $P$ may take so that the vertices of
$P$ (and in general the codimension $2$ faces of $P$) never pass
through the discrete set of translations vectors $\Lambda$, shown
here as a lattice.} \label{moving.polygon}
\end{figure}

\newpage
\section{An analytic approach, using Fourier techniques}\label{FourierSection}

In this section we give another proof of the main result, Theorem
\ref{thm_main}, but this time from the Fourier perspective, so that
we may employ the language of generalized functions. The reader may
consult  the classic reference \cite{Stein} for more information about Fourier analysis on Euclidean spaces.   We begin once again with
the definition of a $k$-tiling.  Thus, we suppose that a polytope
$P$
$k$-tiles $\R^d$ with some discrete multiset $\Lambda$.  In other
words, we assume that
$$
\sum\limits_{\lambda\in\Lambda}1_{P+\lambda}(v)=k,
$$
for all ${v}\notin \partial P+\Lambda$.   We can rewrite this
condition as a convolution of generalized functions, as follows:
\begin{equation}\label{convolution}
1_P * \delta_\Lambda = k,
\end{equation}
where $1_P$ is the indicator function of $P$, and where
$\delta_\Lambda :=  \sum_{\lambda \in \Lambda} \delta_\lambda$,
where $\delta_\lambda$ is the unit point mass for the point $\lambda
\in \Lambda$.   That is, $\delta_\lambda$ equals $1$ at the point
$\lambda$ and zero elsewhere.    We first differentiate both sides
of (\ref{convolution}), with respect to any $\xi \in \R^d$,
obtaining
\begin{equation}\label{Step1}
\frac{d}{d \xi} \left( 1_P * \delta_\Lambda \right)= (\frac{d}{d
\xi} 1_P) * \delta_\Lambda = 0.
\end{equation}

\noindent Next, we take the Fourier transform of both sides of
(\ref{Step1}), obtaining
\begin{equation}\label{Step2}
\left( \xi \hat 1_P \right) \, \hat \delta_\Lambda = 0,
\end{equation}
where the last step uses the standard Fourier identities $\widehat{
\left( \frac{d}{dx} F \right)}(\xi) = \xi \hat F(\xi)$, and
$\widehat{F*G} = \hat F \hat G$. If we now have some more detailed
knowledge about $\hat 1_P$, then we can use (\ref{Step2}) to proceed
further.

The next result is a useful combinatorial version of Stokes'
formula, which holds for the Fourier transform of the indicator
function of any polytope.  This is a result about $\hat 1_P$ that
appears to be not as well-known, so we prove it in complete detail.
For the transform of a function on  $\R^d$, we use the standard
definition:
\[
\hat 1_P(\xi) := \int_P \expo(2\pi i \langle \xi, x \rangle ) dx,
\]
valid for any $\xi \in \R^d$, because $P$ is compact.

\begin{theorem}\label{Stokes}
Let $F$ be a $k$-dimensional polytope in $\mathbb R^d$, for any $k
\leq d$. Let $Proj_F(\xi)$ denote the orthogonal projection of $\xi$
onto the $k$-dimensional subspace of $\mathbb R^d$ that is parallel
to $F$. Moreover, for each $(k-1)$-dimensional face $G \in \partial
F$, let $n_G$ be its outward pointing normal vector. Then the
Fourier transform of the indicator function of $F$ can be written as
follows:
\medskip

Case I.  \   \ If $Proj_F(\xi) = 0$, then
\[
\hat 1_F(\xi) = V^k(F) {\expo}(2 \pi i \Phi ),
\]
where $\Phi$ is the constant value of the function $\phi(x) :=
\langle \xi, x \rangle$ on $F$.

\medskip
Case II.  \  \ If $Proj_F(\xi) \not= 0$, then
\[
\hat 1_F(\xi) = -\frac{1}{2\pi i} \sum_{G \in \partial F}
\frac{\langle Proj_F(\xi), n_F \rangle}{|| Proj_F(\xi)||^2 } \hat
1_G(\xi).
\]
\end{theorem}
\begin{proof}
We note that the gradient of $\phi(x):= \langle \xi, x \rangle$,
with respect to the Riemannian structure of the submanifold $F
\subset \R^d$, is simply the projection of the $d$-dimensional
Euclidean gradient of $\phi$ onto $F$.  We denote this projection by
$\textup{grad}_F\phi$ in the argument that follows.  Fix any $\xi
\in \R^d$.

Case I.  \  If $Proj_F(\xi) = 0$, then $\textup{grad}_F\phi = 0$, so
that $\phi$ is constant on $F$. The Fourier integral defining $\hat
1_F(\xi)$ in this case degenerates into an integral of a constant
function on $F$, hence the conclusion of the theorem for this case.

Case II.  \ If $Proj_F(\xi) \not= 0$, then from the linearity of
$\phi$ it follows that $\textup{grad}_F\phi(x) = 2\pi Proj_F(\xi)$
is a constant vector field on $F$. The identity
\[
\textup{div}_F \textup{grad}_F exp(2\pi i \phi(x)) = (2\pi i)^2
||\textup{grad}_F \phi(x)||^2 exp(2\pi i \phi(x))
\]
shows us that $\textup{exp}(2\pi i \phi(x))$ is an eigenfunction of
the Laplacian, with eigenvalue
\[
\lambda := (2\pi i)^2 ||\textup{grad}_F\phi||^2 \not=0.
\]
Hence
\begin{align*}
        \hat 1_F(\xi) &=  \int_F e^{2\pi i \phi(x)} \\
         &=  \frac{1}{\lambda}
             \int_F \textup{div} \left(\textup{grad}_F e^{2\pi i \phi(x)} \right) dF \\
         &=  \frac{1}{\lambda} \sum_{G \in \partial F} \int_G
              \langle \ \textup{grad}_F e^{2\pi i \phi(x)}, n_G \ \rangle  dG,
    \end{align*}
where we've used the identity for the Laplacian above in the second
equality, and Stokes' theorem for the polytope $F$ and its finite
collection of boundary polytope components $G \in \partial F$ in the
third equality. Unravelling the remaining definitions, we get:
\begin{align*}
        \hat 1_F(\xi) &=  \frac{2\pi i}{\lambda} \sum_{G \in \partial F}
       \langle \textup{grad} \phi(x), n_G \ \rangle  \int_G  e^{2\pi i \phi(x)} dG \\
          &=  -\frac{1}{2\pi i} \sum_{G \in \partial F}
           \frac{ \langle Proj_F (\xi), n_G \rangle}{ ||Proj_F(\xi)||^2 }  \hat 1_G(\xi).
\end{align*}
\end{proof}

The result above uses functions, as opposed to generalized
functions, but we may indeed pass to generalized functions, abusing
the notation $\hat 1_P$ only slightly.

Applying Theorem (\ref{Stokes}) above to the generalized function
$1_P$, we may continue from (\ref{Step2}) to get the identity
\begin{equation}\label{Step3}
\left( \sum_{F \in \partial P}  \xi \frac{ \langle \xi, n_F  \rangle
}{\langle \xi, \xi \rangle} \hat 1_F \right) \, \hat \delta_\Lambda
= 0,
\end{equation}
valid for any nonzero $\xi \in \R^d$.   We also note that the sum
runs over all the (codimension $1$) facets $F$ of the boundary
$\partial P$.  It now follows, upon taking the inner product with
$\xi$, that
\begin{equation}\label{Step4}
\left( \sum_{F \in \partial P}   \langle \xi, n_F \rangle \, \hat
1_F  \right) \, \hat \delta_\Lambda = 0.
\end{equation}
Taking Fourier transforms again, we may rewrite the last equation as
\begin{equation}\label{Step5}
\left( \sum_{F \in \partial P}  \left(   \frac{d}{d \,n_F}   \right)
( 1_F ) \right) * \delta_\Lambda = 0.
\end{equation}

We now focus our attention on each {\it pair} of facets of $P$, as
in the first section.  Thus, we consider a facet $F^+$ with its
outward pointing normal $n(F)$, and a parallel facet $F^-$, with its
outward pointing normal $-n(F)$.

\begin{lemma}\label{FacetPairs}
For each facet $F$ of $P$, we have the identity
\[
  \left(  1_{F^+} - 1_{F^-}  \right) * \delta_\Lambda = 0.
 \]
\end{lemma}
\begin{proof}
We assume that  $ \left(  1_{F^+} - 1_{F^-}  \right) *
\delta_\Lambda \not= 0$.  Therefore there exists a small ball $B_r$,
of radius $r$,  such that for any nonnegative, nonzero test
function $f$ whose support is contained in $B_r$, we have $  \langle
\left(  1_{F^+} - 1_{F^-}  \right) * \delta_\Lambda, f \rangle \not=
0$. We may further assume that the support of $f$ is disjoint from
the support of $ \left(  1_{G^+} - 1_{G^-}  \right) *
\delta_\Lambda$, for any facet $G$ of $P$ where $G \not= F$. Indeed,
the discreteness of $\Lambda$ guarantees that we can find such a
ball $B_r$ on which $f$ satisfies the above conditions.

Now we construct a test function $g$ whose support is contained in
$B_r$, with positive derivative $\left( \frac{d}{d n_F} \right)$
along the direction $n_F$, in a small $\epsilon$ vicinity of $B_r
\cap Supp( \left(  1_{F^+} - 1_{F^-}  \right) * \delta_\Lambda ) :=
D_\epsilon$.    To construct such a $g$, we first restrict $f$ to
 $D_\epsilon$, call it $f_0$.   We now multiply $f_0$ by a one dimensional smooth bump function $b$ whose derivative on $[-\epsilon, \epsilon]$ is
 positive, and whose support lives in $[-2\epsilon, 2\epsilon]$.  Thus $g := f_0 \cdot b$ has positive derivative on $D_\epsilon$.   When we
 insert this $g$ into (\ref{Step5}), we arrive at a contradiction. Indeed $<\left(  1_{G^+} - 1_{G^-}  \right) * \delta_\Lambda,\frac{d}{d n_G} g>=0$ for $G \not= F$ because the choice of the support of $g$. On the other hand  $<\left(  1_{F^+} - 1_{F^-}  \right) * \delta_\Lambda, \frac{d}{d n_F} g>\not=0$ by the construction, since $\frac{d}{d n_F}g$ is positive in the vicinity of the support of $\left(  1_{F^+} - 1_{F^-}  \right) * \delta_\Lambda$.


\end{proof}

We finish this section by remarking that iteration of Lemma \ref{FacetPairs} allows us to establish the same conclusion as Lemma \ref{lm_discrete}.   The next iteration would be applied to a normal vector to a facet of $F^+$ within the affine span of the facet $F^+$.    The Fourier analogue of Lemma \ref{lm_integral} involves the scalar product of $ \left(  1_{F^+} - 1_{F^-}  \right) * \delta_\Lambda$ against an ``approximate identity'' function, compactly supported on a large ball.   The main Theorem \ref{thm_main} now follows in a similar manner as in the previous section.   


\section{Another equivalent condition for $k$-tiling, using solid angles}\label{solid-angles}
Here we show that it is possible to reinterpret the condition that a
polytope $k$-tiles $\R^d$ by considering all of the solid angles
$\omega_{P}(\lambda)$ of the $d$-dimensional convex polytope $P$, at
each point $\lambda \in \Lambda$.  For any point $\lambda \in \R^d$, we define
the solid angle at $\lambda$ to be the proportion of a small sphere of radius $R$, centered at $\lambda$, which intersects $P$.
More precisely, the solid angle is defined by 
\[
\omega_{P}(\lambda)=\lim\limits_{R\rightarrow 0}
\frac{V\left(\{\lambda+B_R)\}\cap P\right)}{V(B_R)},
\]
where $V(S)$ is the $d$-dimensional volume of $S$. The following
Theorem is of independent interest, showing another interesting
equivalent condition for $k$-tiling Euclidean space.

\begin{theorem}
A polytope $P$ $k$-tiles $\R^d$ with the multiset $\Lambda$  if and only if
\[
\sum\limits_{\lambda\in{\Lambda}}\omega_{P+v}(\lambda)=k,
\]
for every $v \in \R^d$.
\end{theorem}
\begin{proof}
Suppose that $P$ $k$-tiles $\R^d$ with the multiset $\Lambda$.   We know from Theorem \ref{thm_main} that $P$ must be centrally symmetric, and therefore 
 $-P$ $k$-tiles as well, with the multiset $\Lambda$.  By Lemma \ref {lm_restate}   $\#(\Lambda\cap\{P+x\})=k$ for
almost every $x\in\R^d$.   We can therefore integrate this equality in the variable $x$,
over a $d$-dimensional ball $ B_R(v)$ with center in $v$ and
radius $R$, as follows:

\begin{align*}
k \cdot V(B_R(v))=\int\limits_{B_R(v)}k \,dx  &=
\int\limits_{B_R(v)}\#(\Lambda\cap\{P+x\})dx \\
&= \int\limits_{B_R(v)}\sum\limits_{\lambda\in\Lambda}1_{\lambda-P}(x)dx \\
&= \sum\limits_{\lambda\in\Lambda}\int\limits_{B_R(v)}1_{\lambda-P}(x)dx \\
&=\sum\limits_{\lambda\in\Lambda}V(B_R(v)\cap\{\lambda-P\}) \\
&=\sum\limits_{\lambda\in\Lambda}V(\{\lambda-B_R\}\cap\{P+v\})
\end{align*}

\noindent
It follows that
$k=\sum\limits_{\lambda\in\Lambda}\frac{V(\{\lambda-B_R\}\cap\{P+v\})}{V(B_R(v))}$,
which approaches
$\sum\limits_{\lambda\in\Lambda}\omega_{P+v}(\lambda)$ as $R$ goes
to $0$.  

\medskip
\noindent 
In the other direction, the assumption that $\sum\limits_{\lambda\in{\Lambda}}\omega_{P+v}(\lambda)=k$ is, in general position, equivalent to
the statement that  $\#(\Lambda\cap\{P+x\})=k$.   By Lemma \ref {lm_restate} we conclude that $-P$ $k$-tiles with the multiset $\Lambda$.
Finally, by Theorem \ref{thm_main} we know that $P$ is centrally symmetric, so that $P$ $k$-tiles with the same multiset $\Lambda$. 
\end{proof}

\medskip
We note that a particularly interesting choice of $v$ in this
Theorem is the value $v=0$, so that we can in fact have points in
$\Lambda$ coincide with vertices of $P$.   This equivalent condition
allows us to consider such coincidences without having to translate
$P$ into general position.


\section{Some open questions}\label{open}

\bigskip
We conclude our paper with some fascinating open questions which the main results of the present paper suggest as a natural 
 research direction for $k$-tilings, a relatively new area.  
\bigskip

 \bigskip
 \noindent
1.   \ Recall that the  Venkov-McMullen condition for the existence of belts consisting of $4$ or $6$ parallel codimension $2$ faces allowed an
  ``if and only if" characterization for $1$-tiling polytopes.   Find the analogous additional condition that would give a complete characterization for 
  $k$-tiling polytopes.

 \medskip
 \noindent
2.  \ Classify the combinatorial types of all polytopes which $k$-tile $\R^d$ by translations.

\medskip
We note that for the classical question of $1$-tiling $\R^d$ by
parallelotopes (and parallelotopes are the only objects that can
tile $\R^d$, by McMullen's theorem), there are exactly $5$
combinatorially distinct parallelotopes in $\R^3$, and exactly $52$
distinct parallelotopes in $\R^4$.  It is still not known how many
combinatorially distinct parallelotopes there are in dimensions $5$
and higher.    It is also not known how many facets a parallelotope
may have in general (see \cite{Gru} for references).

\medskip
\noindent
3.  \  Prove or disprove that if any polytope  $k$-tiles $\R^d$ by
translations, then it also $m$-tiles $\R^d$ by a lattice, for a
possibly different $m$.

\medskip
 This would give an analogue of the McMullen Theorem for $1$-tiling parallelotopes in $\R^d$, but appears to be a very difficult problem.

\bigskip
 \noindent
 4.  \ Prove or disprove  that if a $3$-dimensional polytope, which is not a prism, $k$-tiles $\R^3$ by translations with a 
 multiset $\Lambda$,  then $\Lambda$ is a union of a finite number of $3$-dimensional lattices.

\medskip
 This would prove the $3$-dimensional analogue of Kolountzakis' $2$-dimensional result \cite{Kolsurvey}.  

\bigskip
 \noindent
 5.   \ Is it always true that whenever $P$ $k$-tiles with a multiset $\Lambda$, it follows that $\Lambda + v = \Lambda$ for some
 $v \in \R^2$ ?  (This is one of Kolountzakis' open questions in \cite{Kol2})

  \bigskip \noindent
 6.  \ Find, or estimate, the smallest $k$ for which a given polytope can $k$-tile $\R^d$.  This problem is open even in two dimensions.

\end{document}